\documentclass{article}
\date{\today}
\title{The second secant zeta function and its anti-periodization evaluated at $1/\sqrt{n}$} 
\author{Bruno~D.~Welfert\thanks{School of Mathematical and Statistical Sciences, Arizona State University, Tempe AZ 85287, USA, welfert@asu.edu}}
\usepackage{rotating}
\usepackage{bm,amsmath,amssymb,amsfonts}
\usepackage{MnSymbol}
\DeclareMathVersion{normal2}
\usepackage{tikz}
\usepackage{siunitx}
\mathversion{normal2}
\mathversion{normal}
\usepackage{multirow}
\usepackage{comment}
\usepackage{alltt}
\usepackage{listings}
\graphicspath{{./Figs/}}

\newcommand{\todo}[1]{{\color{red}#1}}

\newtheorem{theorem}{\bf Theorem}[section]
\newtheorem{conjecture}{\bf Conjecture}[section]
\newtheorem{remark}{Remark}[section]
\newenvironment{proof}{\textbf{Proof}.}{$\blacksquare$\\[10pt]}
\newenvironment{ack}{\textbf{Acknowledgements}.}{}
\newcommand{\rj}[1]{$\dfrac{1}{\sqrt{#1}}$}
\newcommand{\pj}[2]{$\dfrac{#1}{#2}$}

\begin{document}
\maketitle




\begin{abstract}
The 2-periodic secant zeta function $\psi(r)$ and its $1$-antiperiodization 
$f(r)=(\psi(r)-\psi(r+1))/2$, arising from a fluid dynamics application, are
investigated. In particular, their values at $1/\sqrt{n}$ for positive integers $n$ are determined, using a reformulation of an identity satisfied by $\psi$ as an Abel equation.
\end{abstract}

\section{Introduction}

The primary goal of this study is the evaluation of
\begin{equation}\label{f}
f(r) = \frac{4}{\pi^2}\sum_{k\ge0} \frac{\sec[(2k+1)\pi r]}{(2k+1)^2},
\end{equation}
at $r=1/\sqrt{n}$ for integer $n$.
The Dirichlet series \eqref{f} arises in a recent study of the response of a vertically stratified fluid inside a two-dimensional square cavity to horizontally forced harmonic perturbations \cite{GYWL20}.
In that study, $r$ is a parameter controlling the frequency of
oscillations in the low amplitude forcing, with values $r\in(0,1/2)$ associated with the existence of internal waves, while
$f(r)$ characterizes the velocity, pressure and temperature response, 
in the limits of low viscosity and low amplitude forcing,
along characteristic lines emanating from the midpoints of both 
vertical sides of the cavity, and somewhat controls 
the response elsewhere. In particular, resonance is shown in \cite{GYWL20} 
to occur at precisely the singularities $r=p/(2q)$ of $f$, with both $p$ and $q$ odd integers, while the response at other values rational values $r$ has piecewise linear velocity and temperature fields, in agreement with the $1/(2k+1)^2$ behavior of the terms in the series \eqref{f}. 
The specific value $f(1/\sqrt{8})$ is used in \cite{GYWL20}
to illustrate a self-similar behavior of the response at a frequency associated with a quadratic irrational $r$. The objective of this work is to evaluate $f$ at general $1/\sqrt{n}$ for integer $n$ and, in the process, get a better understanding of the self-similar behavior observed in the fluid problem.

\maketitle

\noindent
The function $f$ is related to the secant zeta function 
\begin{equation}\label{psi}
\psi(r) = \frac{4}{\pi^2}\sum_{n\ge1} \frac{\sec[n\pi r]}{n^2}
\end{equation}
introduced in \cite{LRR14} (where it is denoted $\psi_2$, hence ``second'', 
without the factor $\pi^2/2$) via
\begin{equation}\label{relation2}
f(r)=\frac{\psi(r)-\psi(r+1)}{2}=-f(r+1).
\end{equation}
Equation \eqref{relation2} shows that $f$ is $1$-antiperiodic, and 
emphasizes how it can be constructed from the $2$-periodic function 
$\psi$ via a process we call anti-periodization. 
Another useful relation,
\begin{equation}\label{relation}
f(r) = \psi(r)-\frac{\psi(2r)}{4}.
\end{equation}
can be verified directly and is equivalent to \cite[Equation 3.3]{LRR14}.
Together with the symmetry
\begin{equation}\label{fsym}
f(-r)=f(r),
\end{equation}
equation \eqref{relation} also implies anti-symmetry about $r=1/2$, i.e., 

\begin{equation}\label{fantisym}
f(1-r)=f(-r+1)=-f(-r)=-f(r),
\end{equation}
which can be restated as
\begin{equation}\label{fantisym2}
f(1/2-s)=-f(1/2+s)
\end{equation}
for $s=r-1/2$. 
Thus, $f$ is also an anti-symmetrization of $\psi$. 
The values
\begin{equation}\label{psi0f0}
\psi(0) = \frac{4}{\pi^2}\sum_{n\ge0}\frac{1}{n^2} = \frac{2}{3}
\quad\text{and}\quad
f(0)=\frac{4}{\pi^2}\sum_{k\ge0}\frac{1}{(2k+1)^2}=\frac{1}{2}
=-f(1)
\end{equation}
follow from standard results.
Since $\psi$ is even, $\psi(-r)=\psi(r)$, and $2$-periodic,  
it is sufficient to restrict $r$ to $(0,1)$ to evaluate $\psi$.
In view of \eqref{relation2} and \eqref{fantisym} it is sufficient to consider the range $r\in(0,1/2)$ for $f$. 
Figure~\ref{fig:fandpsi} shows graphs of $\psi$ and $f$ for a selection of \num{2e5} values $r$ uniformly distributed in the range $(0,1)$. 
The series \eqref{f} and \eqref{psi} are evaluated using \num{e5} terms each.
The graph of $f$ is verified to be symmetric about the point $(1/2,0)$,
according to \eqref{fantisym2}. The values $f(r)$ are also observed to be at or near $1/2$ for many values $r$ away from $r=1/2$.

Both $\psi$ and $f$ are discontinuous everywhere, having singularities at rational values $r=(2p+1)/(2q)$ for any integers $p$ and $q\ne0$, with the restriction that $q$ be odd for $f$. 
The fact that $f$ ``only'' has half as many singularities as $\psi$, combined with the filtering property of the periodization \eqref{relation2} ($f(r)=i\sin(\pi r)\psi(r)$ for $\psi(r)=\exp(i\pi r)$), makes the graph of $f$ appear, in some sense, better behaved.

\begin{figure}
\begin{center}
\includegraphics[width=.9\linewidth]{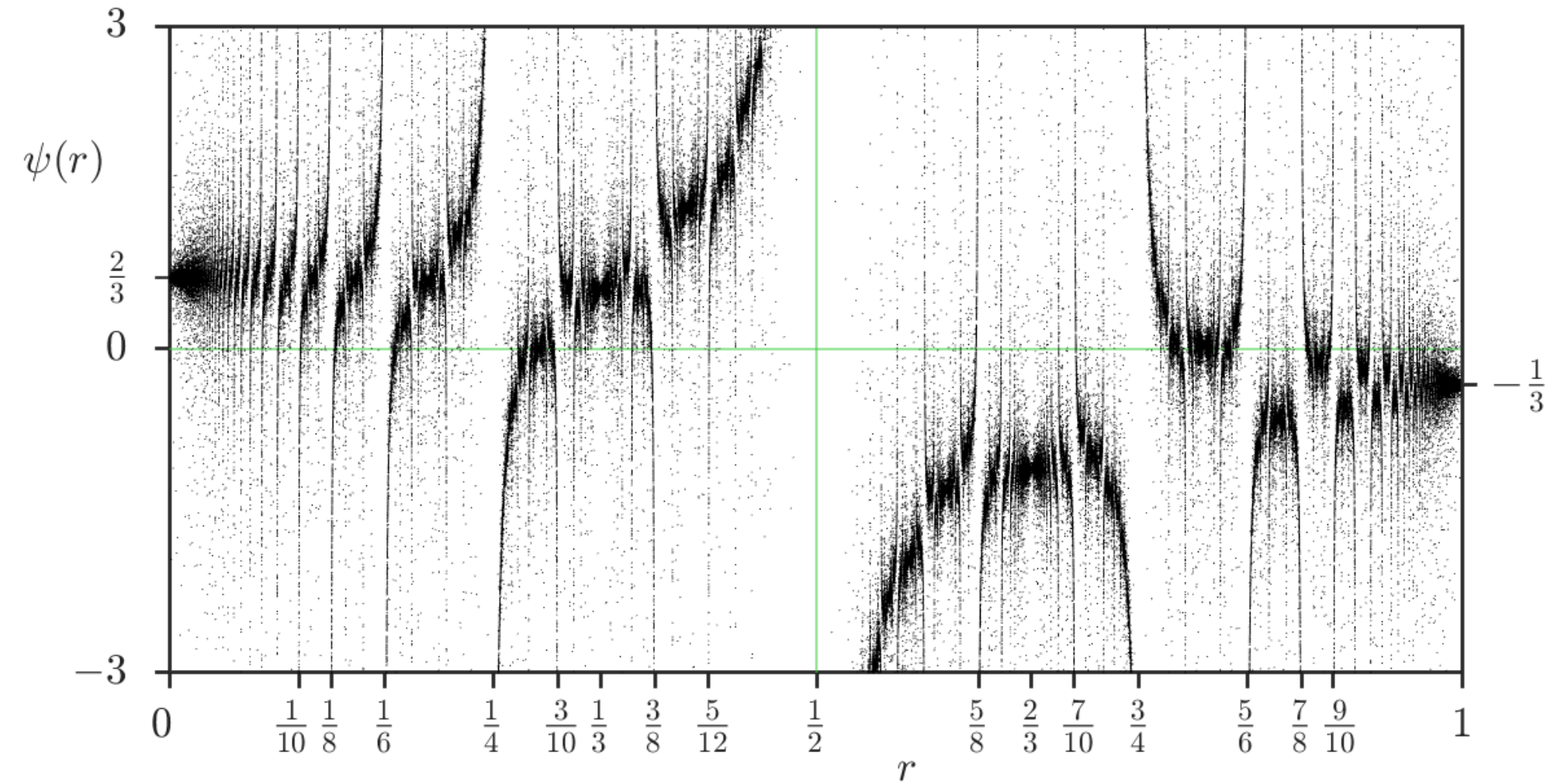} \\
\includegraphics[width=.9\linewidth]{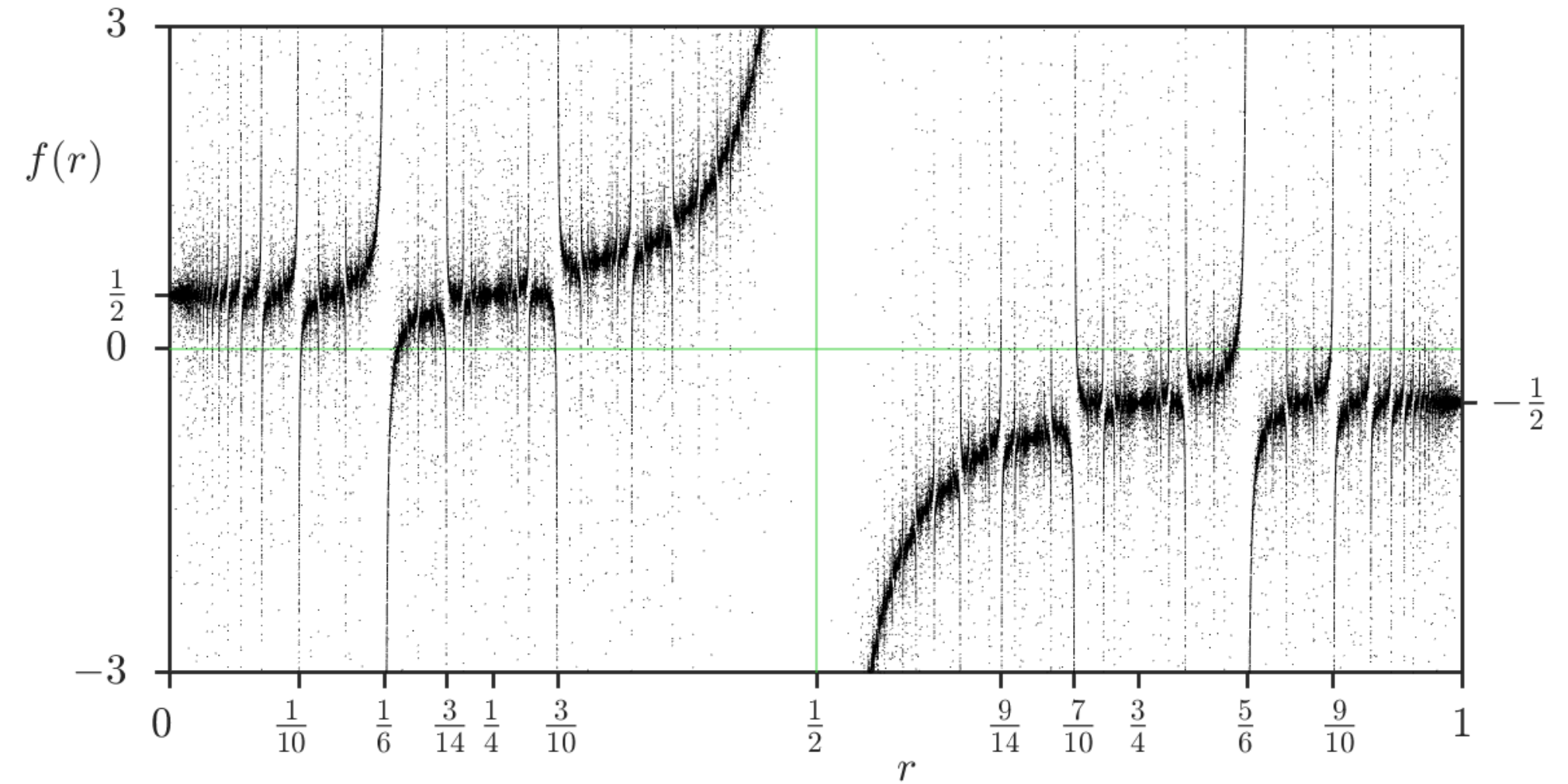}
\caption{Functions $\psi$ (top) and $f$ (bottom) evaluated at \num{2e5} uniformly distributed random numbers $r\in(0,1)$ using \num{e5} terms
in the series \eqref{f} and \eqref{psi}. See Appendices~\ref{appA} and \ref{appC} for numerical evaluations of the series using \textsc{Matlab}.}
\label{fig:fandpsi}
\end{center}
\end{figure}

Certain values of $f$ for $r\in(0,1/2)$ can be readily obtained from 
known values of $\psi$ using \eqref{relation} and the $2$-periodicity of $f$. For example,
\begin{equation}\label{frppp1}
f(r) = f(r+2p) =
f\left(\sqrt{2p(2p+1)}\right) = 2p+\frac{2}{3}-\frac{1}{4}\cdot\frac{2}{3} 
=  2p+\frac{1}{2}
\end{equation}
for integers $p$ at $r=1/[1+\sqrt{1+1/(2p)}]=\sqrt{2p(2p+1)}-2p$ is obtained by combining results from \cite[page 202]{LRR14}.
However, none of the results from \cite{LRR14}, \cite{Str16}, or \cite{BeSt16} addresses values of $\psi$ at $r=1/\sqrt{n}$, 
including $1/\sqrt{2}$ and $1/\sqrt{8}$, needed to evaluate $f$ at 
the latter via \eqref{relation}.
An explicit expression of $\psi(r)$ in terms of values of the first four Bernoulli polynomials can be found in \cite[Eq. 5]{ChGr14}. However, this expression relies on the identification of a matrix $V\in\text{SL}_2(\mathbb{Z})$ (with integer entries and determinant $1$), which for $r$ of the form $1/\sqrt{n}$ only seems possible (without exploiting periodicity) for $n$ equal to the product of two consecutive integers, $p(p+1)$. Even then, the expression is complicated, and was merely used in \cite{ChGr14} to prove its rationality; see Appendix \ref{appA} for details.

It is shown here how the identity
\cite[Equation 3.8 in Theorem 3 with $k=2$]{LRR14} 
\cite[Equation 3 in Theorem 2.1 with $m=1$]{BeSt16}
\begin{equation}\label{eq3ref}
(1+\tau)\psi\left(\frac{\tau}{1+\tau}\right)
-(1-\tau)\psi\left(\frac{\tau}{1-\tau}\right)
=4\cdot\frac{\tau(2+\tau^2)}{6(1-\tau^2)}
\end{equation}
can be leveraged to achieve that goal, and show in particular that 
\begin{equation}\label{fresult}
f(1/\sqrt{8}) = 1.
\end{equation}
The right-hand side of \eqref{eq3ref} is, 
up to a factor $4$ introduced by the scaling in \eqref{psi},
the coefficient of $z$ in the Laurent series of the function 
$\sin[\tau z]/(\sin[(1+\tau)z]\sin[(1-\tau)z])$ of $z$.

The approach followed here uses, in essense, the same ideas based on modular representation adopted in \cite{LRR14} and \cite{BeSt16}.
The starting point, however, is a reformulation of \eqref{eq3ref}
as a classical Abel functional equation \cite[\S3.5C]{KCG90}
\begin{equation}\label{Abel}
\Phi(\varphi(r))-\Phi(r)=1
\end{equation}
with
\begin{equation}\label{rphig}
r=\frac{\tau}{1-\tau}, \quad \varphi(r)=\frac{r}{1+2r},
\quad\text{and}\quad \Phi(r)=\frac{3}{4}\left(\frac{\psi(r)}{r}+r\right),
\end{equation}
which provides a simpler path to explicit values of $\psi(r)$
and $f(r)$ at $r=1/\sqrt{n}$, including \eqref{fresult}.
Equation \eqref{Abel} has a long history, with solutions which can be locally continuous or completely (everywhere) discontinuous depending on the definition of the map $\varphi$, with the latter being the case here for the map $\varphi$ in \eqref{rphig}. Note that $\Phi(0)=\infty$ at the only fixed point $r=0$ of this $\varphi$. The Abel equation \eqref{Abel} can be interpreted as a discrete form of linearity for the function $\Phi$: it implies
\begin{equation}\label{discreteder}
\Phi(\varphi(r))-2\Phi(r)+\Phi(\varphi^{-1}(r))=0,
\end{equation}
whose left-hand side is a nonstandard second-order difference approximation
of the second derivative of $\Phi(r)$ on the non-uniform grid $[\varphi(r),r,\varphi^{-1}(r)]=[r/(1+2r),r,r/(1-2r)]$ ($r>0$),
which can be verified to be equivalent to a standard differencing of $r\Phi(r)$ on the same grid. It can also be turned into the homogeneous form
\begin{equation}\label{Abel_homogeneous}
\widetilde{\Phi}(\varphi(r))-\widetilde{\Phi}(r)=0
\end{equation}
for the function
\begin{equation}\label{Phitilde}
\widetilde{\Phi}(r) = \Phi(r)-\frac{1}{2r} 
= \frac{3}{4}\left(\frac{\Psi(r)-2/3}{r}+r\right)
\end{equation}

\S\ref{sec2} discusses how \eqref{Abel} 
can be exploited to obtain values $\psi(1/\sqrt{n})$ for certain integers $n$.  
A similar Abel equation for a function related to $f$ is derived in \S\ref{sec3} by leveraging the connection \eqref{relation2} between $f$ and $\psi$, leading to values $f(1/\sqrt{n})$, including \eqref{fresult}. Additional comments and extensions are discussed in
\S\ref{sec4}, which also includes a conjecture, supported by numerical evidence, addressing zeros of $f$ of the form $1/\sqrt{n}$.

\section{Values $\psi(1/\sqrt{n})$}
\label{sec2}

We first show equation \eqref{Abel} holds. One easily verifies that
\[ \varphi(r)=\frac{\tau}{1+\tau} \quad\text{and}\quad
\varphi(r)-r = -\frac{2\tau^2}{1-\tau^2}. \]
Dividing \eqref{eq3ref} throughout by $\tau$ gives
\[ \frac{\psi(\varphi(r))}{\varphi(r)}-\frac{\psi(r)}{r} 
= \frac{2}{3}\cdot\frac{2-2\tau^2+3\tau^2}{1-\tau^2}
=\frac{4}{3}-(\varphi(r)-r).
\]
Moving the quantity $\varphi(r)-r$ to the left-hand side, using the definition of $\Phi$ in \eqref{rphig} and scaling by the factor $4/3$ then lead to \eqref{Abel}. 
Equation \eqref{Abel} allows a straightforward evaluation of $\Phi$, 
and thus $\psi$, at iterated values
\begin{equation}\label{rk}
r_k = \varphi(r_{k-1}) = \varphi^k(r_0) = \frac{r_0}{1+2kr_0}, \quad k=1,2,\dots
\end{equation}
in terms of $\Phi(r_0)$, i.e., $\varphi(r_0)$.
Indeed, summing up \eqref{Abel} applied to $r=r_k$ for $k=1,\ldots,K$ for some
$K>0$ yields
\begin{equation}\label{eq0}
\Phi(r_K)-\Phi(r_0)=K,
\end{equation}
while summing up \eqref{Abel} applied to $r=r_k$ for $k=1,\ldots,2K$ yields
\begin{equation}\label{eq1}
\Phi(r_{2K})-\Phi(r_0)=2K.
\end{equation}
One easily verifies that 
\begin{equation}\label{rkinv}
\frac{1}{r_k} = \frac{1}{r_0}+2k
\end{equation}
for any $k$ and
\begin{equation}\label{means}
\frac{1}{r_K} = \frac{1}{2}\left(\frac{1}{r_0}+\frac{1}{r_{2K}}\right),
\quad
\Phi(r_K) = \frac{\Phi(r_0)+\Phi(r_{2K})}{2},
\end{equation}
i.e., $r_K$ is the harmonic mean of $r_0$ and $r_{2K}$, while $\Phi(r_K)$ is the arithmetic mean of $\Phi(r_0)$ and $\Phi(r_{2K)}$.
If we select $r_0$ such that
\begin{equation}\label{r0rk}
r_0-r_{2K}=r_0-\frac{r_0}{1+4Kr_0}=2p
\end{equation}
for some $p\in\mathbb{Z}$, $p\ne0$, then the solution
\begin{equation}\label{r0}
r_0 = p+p\sqrt{1+\frac{1}{2pK}}
\end{equation}
of \eqref{r0rk} is such that
\begin{equation}\label{rKinv}
\frac{1}{r_K}=\frac{1}{r_0}+2K = 2K\sqrt{1+\frac{1}{2pK}} = 
\sqrt{\frac{2K(2pK+1)}{p}}
\end{equation}
is the square root of an integer if $p|2K$, precisely the type of values $r$ at which we are interested in evaluating $\psi$. For this $r_0$, we also have
\begin{equation}\label{r0+r2K}
\frac{r_0r_{2K}}{r_K} = \frac{r_0+r_{2K}}{2} = r_0-p
= p\,\sqrt{1+\frac{1}{2pK}} = \frac{p}{2Kr_K}.
\end{equation}
Moreover, the 2-periodicity of $\psi$ implies
$\psi(r_{2K})=\psi(r_0)$, i.e.,
\begin{equation}\label{eq2}
\frac{\Phi(r_{2K})}{r_0}-\frac{\Phi(r_0)}{r_{2K}}
=\frac{3}{4}\frac{r_{2K}^2-r_0^2}{r_0r_{2K}}
=-\frac{3p}{2}\left(\frac{1}{r_0}+\frac{1}{r_{2K}}\right)
=-\frac{3p}{r_K}
\end{equation}
using \eqref{rphig}, \eqref{means} and \eqref{r0rk}.
Equations \eqref{eq1} and \eqref{eq2} form a linear system for $\Phi(r_{2K})$ and $\Phi(r_0)$, whose solution can be substituted into \eqref{means} to obtain $\Phi(r_K)$ and thus $\psi(r_K)$. Using \eqref{r0+r2K}, we get
\begin{equation}\label{gr2Kgr0}
\begin{bmatrix}\Phi(r_{2K})\\\Phi(r_0)\end{bmatrix}
=\frac{K}{p}
\begin{bmatrix} r_0 \\ r_{2K} \end{bmatrix}
+\frac{3}{4}(r_0+r_{2K})\begin{bmatrix} 1\\1 \end{bmatrix}
\end{equation}
and
\begin{equation}\label{grK}
\Phi(r_K) = \left(\frac{K}{2p}+\frac{3}{4}\right)(r_0+r_{2K})
=\frac{1}{2r_K}\left(1+\frac{3p}{2K}\right).
\end{equation}
Finally, 
\begin{equation}\label{psirK0}
\psi(r_K) = \frac{4}{3}\,r_K\Phi(r_K)-r_K^2 = \frac{2}{3}\left(1+\frac{3p}{2K}\right)-r_K^2=\frac{2}{3}+\frac{p}{K}-r_K^2.
\end{equation}
\begin{theorem}\label{th2.1}
Let $K>0$, $p\in\mathbb{Z}$, $p\ne0$, and $q=2K(2pK+1)$. Then
\begin{equation}\label{psirK}
\psi\left(\sqrt{p/q}\right) 
= 2/3+p/K-p/q.
\end{equation}
\end{theorem}
Table \ref{table1} lists values of \eqref{psirK}
for several pairs $(K,p)$ such that $q/p$ is integer.
The values \eqref{psirK} were already known to be rational 
\cite[Theorem 3.2]{BeSt16}. The columns associated with $p=\pm1$
are obtained in Appendix \ref{appB} via a modular interpretation 
of the above strategy based an expansion of $\psi$ in terms of periodicized Bernoulli polynomials.

\begin{sidewaystable}
\begin{center}
\renewcommand{\arraystretch}{2}
\medmuskip=0mu
\begin{tabular}{@{}c|c|ccccccccccc@{}}
\multicolumn{1}{c}{}& \raisebox{-8pt}{\tikz[scale=.8]{\draw(0,1)--(1,0);\node at (.25,.3){$K$};\node at (.75,.7){$p$};}\hspace*{-6pt}}
&-6&  -4   &  -3   &  -2   &  -1   &   1   &   2   &   3   &   4   &6\\\hline\hline
& 1 &&       &       &\rj{3} &\rj{2} &\rj{6} &\rj{5} &       &       &\\
$r_K$
& 2 &&\rj{15}&       &\rj{14}&\rj{12}&\rj{20}&\rj{18}&       &\rj{17}&\\
& 3 &\rj{35}&       &\rj{34}&\rj{33}&\rj{30}&\rj{42}&\rj{39}&\rj{38}& &\rj{37}\\[5pt]\hline\hline

& 1 &&       &       &-\pj{5}{3}&-\pj{5}{6}&\pj{3}{2}&\pj{37}{15} &    &   &\\
$\psi(r_K)$
& 2 &&-\pj{7}{5}&       &-\pj{17}{42}&\pj{1}{12}&\pj{67}{60}&\pj{29}{18}&       &\pj{133}{51} &\\
& 3 &-\pj{143}{105}&       &-\pj{37}{102}&-\pj{1}{33}&\pj{3}{10}&\pj{41}{42}&\pj{17}{13}&\pj{187}{114}&    &\pj{293}{111}\\
\end{tabular}
\caption{Roots $r_K$ of integer inverses from \eqref{rKinv} (top) 
and values $\psi(r_K)$ from \eqref{psirK} (bottom) obtained for integer values of $p\ne0$ and $K>0$.}
\label{table1}
\end{center}
\end{sidewaystable}

\begin{remark}
The value 
\[ \psi(r_0)
=\frac{4}{3}r_0\Phi(r_0)-r_0^2
=r_0r_{2K}\left(1+\frac{4K}{3p}\right)
=\frac{2}{3}+\frac{p}{2K}
\]
obtained from \eqref{gr2Kgr0} and \eqref{r0+r2K} is equivalent to a formula appearing in \cite[page 844]{BeSt16}.
It also reduces to \cite[Equation 4.7]{LRR14} when $p=2j$ is even
(i.e., $r_0-r_{2K}$ is a multiple of $4$). 
\end{remark}

\begin{remark}
As $K\to\infty$ for fixed $p$ we recover $\psi(0)=2/3$ in \eqref{psirK},
despite the fact that $\psi$ is nowhere continuous.
\end{remark}

\begin{remark}
Values from Table~\ref{table1} can be combined, using \eqref{relation},
to obtain values of $f$ at $r$ of the form $1/\sqrt{n}$ for certain $n$. For example,
\begin{equation}\label{f12}
f(1/\sqrt{12}) = \psi(1/\sqrt{12})-\frac{\psi(1/\sqrt{3})}{4}
=\frac{1}{12}+\frac{5}{12}=\frac{1}{2},
\end{equation}
or
\begin{equation}\label{f20}
f(1/\sqrt{20}) = \psi(1/\sqrt{20})-\frac{\psi(1/\sqrt{5})}{4}
=\frac{67}{60}-\frac{37}{60}=\frac{1}{2}.
\end{equation}
These values confirm the high frequency of the value $1/2$ among 
the sequence $\{f(1/\sqrt{n})\}_{n\ge1}$. 
See \S\ref{sec3} for additional insight.
\end{remark}

\begin{remark}
The Abel equation \eqref{Abel} can be reformulated as the 
similarity property
\begin{equation}\label{similarity0}
\Pi(\alpha r)
=\Pi(\alpha r+(1-\alpha)0) = \alpha\Pi(r)+(1-\alpha)\Pi(0)
\end{equation}
for the shifted function $\Pi(r)=\psi(r)+r^2=4r\Phi(r)/3$, 
with $\Pi(0)=\psi(0)=2/3$ and $\alpha=1/(1+2r)$, or
\begin{equation}\label{homothety0}
\begin{bmatrix} \alpha r \\ \Pi(\alpha r) \end{bmatrix}
=\begin{bmatrix} 0 \\ \Pi(0) \end{bmatrix} +\alpha\begin{bmatrix} r-0 \\ \Pi(r)-\Pi(0) \end{bmatrix}.
\end{equation}
The graph of $\Pi$ is thus invariant under the homothety of centre $(0,\Pi(0))=(0,2/3)$ and similarity ratio $\alpha<1$. 
The results obtained in \S\ref{sec2} for $\Psi$,
such as \eqref{psirK0}, can be restated, perhaps more naturally, 
in terms of $\Pi$. 
$\varphi$ maps the interval $[1/4,1/2]$ to $[1/6,1/4]$ and, more generally the interval 
$[1/(2k+2),1/(2k)]$ to $[1/(2k+4),1/(2k+2)]$ for $k>0$. This shows that all values $\Pi(r)$, and thus $\psi(r)$, for $r\in(0,1/2)$ can be deduced from those in the interval $[1/4,1/2]$ via the affine transformation \eqref{homothety0}.
\end{remark}

\section{Values $f(1/\sqrt{n})$}
\label{sec3}

We now show how $f$ can also be evaluated at values of the form $1/\sqrt{n}$
for certain integers $n$. Similarly to \eqref{rphig} we introduce
\begin{equation}\label{Phi}
g(r)=\varphi^2(r)=\frac{r}{1+4r}, \quad 
G(r) = \frac{f(r)}{2r}.
\end{equation}
These maps are related to the functions $\varphi$ and $\Phi$ in \eqref{rphig} by
\begin{equation}\label{varphi2Phi}
2g(r) = \varphi(2r)
\end{equation}
and, using \eqref{relation},
\begin{equation}\label{g2G}
G(r) = \frac{2\Phi(r)-\Phi(2r)}{3}.
\end{equation}
From the relation \eqref{Abel} we obtain
\begin{align}
G(g(r))-G(r) 
&=\frac{2}{3}\left(\Phi(g(r))-\Phi(r)\right)-\frac{1}{3}\left(\Phi(2g(r))-\Phi(2r)\right) \label{Abelf}\\
&=\frac{2}{3}\left(\Phi(\varphi^2(r))-\Phi(r)\right)-\frac{1}{3}\left(\Phi(\varphi(2r))-\Phi(2r)\right) \notag\\
&=\frac{4}{3}-\frac{1}{3} = 1.\notag
\end{align}
Therefore, $G$ also satisfies an Abel equation.
This equation can be restated directly in terms of $f$
as the similarity relation
\begin{equation}\label{similarity}
f(\alpha r) = f(\alpha r + (1-\alpha)0) = \alpha f(r) + (1-\alpha)f(0), 
\end{equation}
with $\alpha = 1/(1+4r)$, using $f(0)=1/2$ from \eqref{psi0f0};
compare to \eqref{similarity0}.
Equation \eqref{similarity} means that
\begin{equation}\label{homothety}
\begin{bmatrix} \alpha r \\ f(\alpha r) \end{bmatrix}
=\begin{bmatrix} 0 \\ f(0) \end{bmatrix} +\alpha\begin{bmatrix} r-0 \\ f(r)-f(0) \end{bmatrix},
\end{equation}
i.e., the graph of $f$ is invariant under the homothety with centre $(0,f(0))=(0,1/2)$ and similarity ratio $\alpha=1/(1+4r)<1$. 
The function $g$ maps the interval $[1/6,1/2]$ to $[1/10,1/6]$ and, more generally the interval 
$[1/(2k+1),1/(2k-1)]$ to $[1/(2k+3),1/(2k+1)]$ for $k>0$. This shows that all values $f(r)$ for $r\in(0,1/2)$ can be deduced from those in the interval 
$[1/6,1/2]$ via the simple affine transformation \eqref{homothety}.
This similarity can be used to readily evaluate $f$ at certain values of $r$.
For example, using $r=-1$ in \eqref{similarity} and the $1$-antiperiodicity
of $f$,
\begin{equation}\label{fonethird}
f(1/3) = -\frac{1}{3}f(-1)+\frac{4}{3}f(0)=\frac{5}{3}f(0)=\frac{5}{6}.
\end{equation}
The value $f(1/4)$ is somewhat trickier to obtain from \eqref{similarity}.
Heuristically, letting $r=\infty$ yields
\begin{equation}\label{fonefourth}
f(1/4) = 0\cdot f(\infty)+1\cdot f(0)=f(0)=\frac{1}{2}.
\end{equation}
However, the statement \eqref{fonefourth} lacks mathematical rigor
for two reasons: (i) the meaning of $f(\infty)$ is unclear, even in an average (weak) sense, as $f$ takes on all possible values as $r\to\infty$ by periodicity, and (ii) $f$ is not continuous anywhere. Nevertheless, the above value can independently be shown to be correct; see Appendix \ref{appC} for details.

\begin{figure}
\begin{center}
\includegraphics[width=.9\linewidth]{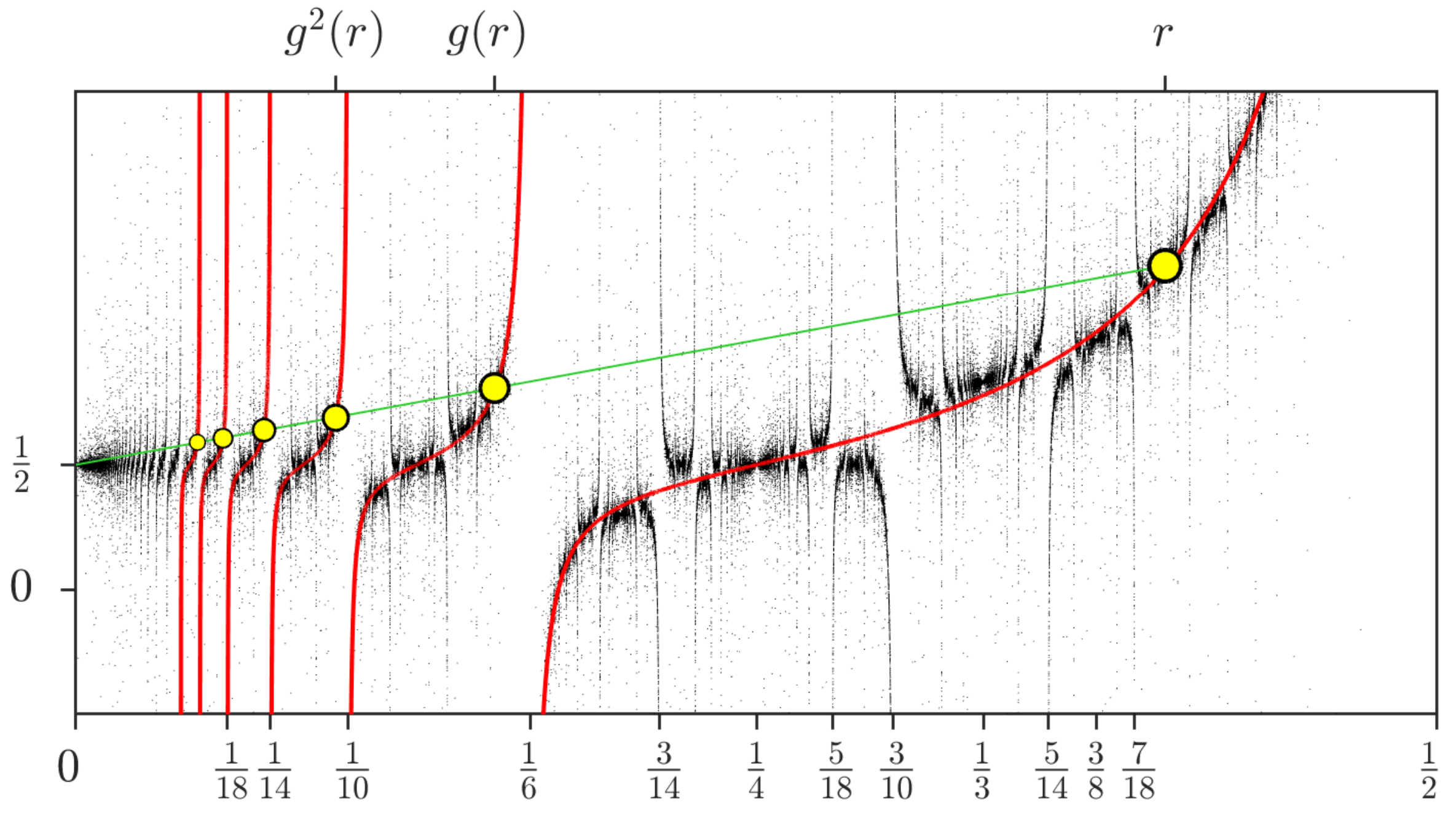}
\end{center}
\caption{Illustration of the invariance of the graph of $f$ under the homothety \eqref{homothety}. The red curves represent the function \eqref{rational_approx} in the interval $r\in(1/6,1/2)$ and its transforms under the homothety \eqref{homothety} for $r<1/6$.}
\label{fig:fsimilar}
\end{figure}
Figure~\ref{fig:fsimilar} illustrates the action of the homothety on the ``approximation''
\begin{equation}\label{rational_approx}
\frac{1}{18}+\frac{1}{36}\;\frac{1}{1-6r}+\frac{1}{4}\;\frac{1}{1-2r}
\end{equation}
of $f$ for $r\in(1/6,1/2)$. The expression \eqref{rational_approx}
matches the asymptotic behavior of $f$ at $r=(1/6)^+$ and $r=(1/2)^-$
(see Appendix \ref{appD} for details) as well as the value the data point \eqref{fonefourth} at the harmonic center of the interval $(1/6,1/2)$.  
Of course, $f$ has (infinitely) many singularities in this interval that are not represented by \eqref{rational_approx}. The graph from figure \ref{fig:fsimilar} shows an accumulation of points $(r,f(r))$
around $r=1/q$, e.g. for $q=3,4$. This behavior can be partly explained 
by noting that
\begin{equation}\label{cont}
f\left(\frac{\ell}{1+4\ell}\right)-f(1/4)
=\frac{f(\ell)+2\ell}{1+4\ell}-1/2
=\frac{(-1)^\ell-1}{2(1+4\ell)}\xrightarrow{\ell\to\infty}0
\end{equation}
using \eqref{similarity}, \eqref{fonefourth} 
and the $1$-antiperiodicity of $f$.
The following result provides an explicit value of $f$ at all such $r$.
\begin{theorem}\label{th3.1}
$f(1/q) = \begin{cases} 
1/2     & \text{if } q\cong 0 \text{ [mod 4]}, \\ 
1/2-1/q & \text{if } q\cong 1 \text{ [mod 4]}, \\ 
\infty  & \text{if } q\cong 2 \text{ [mod 4]}, \\ 
1/2+1/q & \text{if } q\cong 3 \text{ [mod 4]}. \\ 
\end{cases}$ 
\end{theorem}
\begin{proof}
Let $s_0=r$ and $s_k=g(s_{k-1})=s_{k-1}/(1+4s_{K-1})=r/(1+4kr)$ for $k>0$.
Summing-up \eqref{Abelf} applied to $r=s_k$ for $k=1,\ldots,K$ yields
$G(s_K)=G(s_0)+K$, i.e., \eqref{similarity} more generally holds with 
$\alpha=1/(1+4Kr)$ with any integer $K>0$:
\begin{equation}\label{similarity2}
f\left(\frac{r}{1+4Kr}\right) = \frac{f(r)}{1+4Kr}+\frac{2Kr}{1+4Kr}.
\end{equation}
The factor $1/(1+4Kr)$ is the one-step similarity factor 
for the composition of $K$ copies of the homothety $\varphi$ 
with one-step factor $1/(1+4r)$.
Letting $q=4K+\ell$ with $\ell=1,2,3$, or $4$, and applying \eqref{similarity2} with $r=1/\ell$ yields 
\begin{equation}
f(1/q) = \frac{\ell f(1/\ell)+2K}{q} 
= \frac{1}{2}+\left(f(1/\ell)-\frac{1}{2}\right)\frac{\ell}{q}.
\end{equation}
The result then follows from $f(1)=-1/2$, see \eqref{psi0f0},
$f(1/2)=\infty$ (singularity), $f(1/3)=5/6$, see \eqref{fonethird}, and
$f(1/4)=1/2$, see \eqref{fonefourth}.
\end{proof}
Theorem \ref{th3.1} provides values of $f$ at all $1/q=1/\sqrt{n}$ with $q$ positive integer and $n=q^2$. In order to evaluate $f$ for other $n$ we consider, as in \S\ref{sec2}, a sequence
\begin{equation}
s_0=r_0, \quad s_k=g^k(s_0)=\frac{s_0}{1+4ks_0}=\frac{r_0}{1+4kr_0}=r_{2k}
\end{equation}
for $k\ge0$, so that
\begin{equation}\label{invs0s2K}
\frac{1}{s_0}=\frac{1}{s_K}-4K, \quad \frac{1}{s_{2K}}=\frac{1}{s_K}+4K.
\end{equation}
The determination of $f$ at values of the form $1/\sqrt{n}$ can now proceeds as in \S\ref{sec2} for $\psi$, with a little twist: the $1$-antiperiodicity of $f$ can be exploited to relax the constraint, similar to \eqref{r0rk} for the $\{s_k\}_{k\ge0}$ sequence, that $s_0-s_{2K}$ be an even integer.  
We thus select $s_0$ such that 
\begin{equation}\label{r0rkany}
s_0-s_{2K} = p
\end{equation}
for any nonzero $p\in\mathbb{Z}$, with $f(s_{2K})=(-1)^pf(s_0)$.
Substituting $S_{2K}$ in terms of $s_0$ in \eqref{r0rkany} and solving for $s_0$
yields the root
\begin{equation}\label{s0}
s_0 = \frac{p}{2}+\frac{p}{2}\sqrt{1+\frac{1}{2pK}}
\end{equation}
and thus,
\begin{equation}\label{sKs2K}
s_K = \frac{1}{2}\sqrt{\frac{p}{2K(2pK+1)}}, \quad
s_{2K} = -\frac{p}{2}+\frac{p}{2}\sqrt{1+\frac{1}{2pK}}.
\end{equation}
Note that $s_0$ in \eqref{s0} can be obtained from $r_0$ in \eqref{r0} via the 
substitution $p\to p/2$, and that $s_K$ is half of $r_K$ from \eqref{rKinv}. 
From the equation
\begin{equation}
2K=G(s_{2K})-G(s_0),
\end{equation}
obtained by summing up \eqref{Abelf} for $r=s_k$, $k=1,\ldots,2K$, we obtain,
using \eqref{invs0s2K},
\begin{align*}
2K=\frac{f(s_{2K})}{2s_{2K}}-\frac{f(s_0)}{2s_0} 
&= f(s_0)\left(\frac{(-1)^p}{2s_{2K}}-\frac{1}{2s_0}\right) \\
&= f(s_0)\left(\frac{(-1)^p-1}{2s_K}+2K[(-1)^p+1]\right),
\end{align*}
i.e.,
\begin{equation}\label{fs0}
f(s_0)=
\begin{cases} 
1/2 & \text{if $p$ is even}, \\ 
-2Ks_K & \text{if $p$ is odd}.
\end{cases}
\end{equation}
As a result,
\begin{align*}
f(s_K) &= 2s_KG(s_K)=2s_K(G(s_0)+K)\notag\\
&=2s_K\left(\frac{f(s_0)}{2s_0}+K\right)
=\begin{cases} 
\dfrac{2s_K(1+4Ks_0)}{4s_0} = \dfrac{1}{2} & \text{if $p$ is even}, \\ 
2Ks_K\left(1-\dfrac{s_K}{s_0}\right)= 8K^2 s_K^2 = \dfrac{pK}{2pK+1} & \text{if $p$ is odd}. 
\end{cases}
\end{align*}
To summarize,
\begin{theorem}\label{th3.2}
Let $K>0$, $p\in\mathbb{Z}$, $p\ne0$, and $q=2K(2pK+1)$. Then
\begin{equation}\label{fsK}
f\left(\frac{1}{2}\sqrt{p/q}\right)=\begin{cases} 
1/2 & \text{if $p$ is even}, \\ 
1/2-K/q & \text{if $p$ is odd}.
\end{cases}
\end{equation}
\end{theorem}
Theorem \ref{th3.2} confirms the high rate of occurence of 
the value $1/2$ in the sequence $\{f(1/\sqrt{n})\}_{n\ge1}$.
Equation \eqref{fs0} shows that $f$ also takes this value at many 
fully quadratic irrationals.
Table~\ref{table2} illustrates values of $f$ for several pairs $(K,p)$.
In particular, \eqref{fresult} is recovered for $K=1$ and $p=-1$.
Table~\ref{table2} also confirms both \eqref{f12} and \eqref{f20}.

\begin{table}
\begin{center}
\renewcommand{\arraystretch}{2}
\medmuskip=0mu
\begin{tabular}{@{}c|c|cccccc@{}}
\multicolumn{1}{c}{}& \raisebox{-8pt}{\tikz[scale=.8]{\draw(0,1)--(1,0);\node at (.25,.3){$K$};\node at (.75,.7){$p$};}\hspace*{-6pt}}
& -4   & -2   &  -1   &   1   &   2   &   4  \\\hline\hline
\multirow{2}{*}{$s_K$}
& 1 &  &\rj{12} &\rj{8} &\rj{24} &\rj{20} &   \\
& 2 &\rj{60}&\rj{56}&\rj{48}&\rj{80}&\rj{72}&\rj{68}\\[5pt]\hline\hline
\multirow{2}{*}{$f(s_K)$}
& 1 &  &\pj{1}{2}&1&\pj{1}{3}&\pj{1}{2} & \\
& 2 &\pj{1}{2}&\pj{1}{2}&\pj{2}{3}&\pj{2}{5}&\pj{1}{2}&\pj{1}{2}\\
\end{tabular}
\caption{Roots $s_K$ of integer inverses from \eqref{sKs2K} (top) 
and values $f(s_K)$ from \eqref{fsK} (bottom) obtained for integer values of $p\ne0$ and $K=1,2$.}
\label{table2}
\end{center}
\end{table}

\begin{remark}
Using \eqref{relation}, the values \eqref{fsK} can now be combined with \eqref{psirK} to obtain values of $\psi$ at $s_K$:
\begin{equation}\label{psisK}
\psi\left(\frac{1}{2}\sqrt{p/q}\right) = 
f\left(\frac{1}{2}\sqrt{p/q}\right)+\frac{1}{4}\psi\left(\sqrt{p/q}\right) 
= \frac{2}{3}+\frac{p}{4K}-\frac{p}{4q}
-\begin{cases} 0 & \text{if $p$ is even}, \\ K/q & \text{if $p$ is odd.}\end{cases}.
\end{equation}
\begin{table}
\begin{center}
\renewcommand{\arraystretch}{2}
\medmuskip=0mu
\begin{tabular}{@{}c|c|cccccc@{}}
\multicolumn{1}{c}{}& \raisebox{-8pt}{\tikz[scale=.8]{\draw(0,1)--(1,0);\node at (.25,.3){$K$};\node at (.75,.7){$p$};}\hspace*{-6pt}}
& -4   &  -2   &  -1   &   1   &   2   &   4  \\\hline\hline
\multirow{2}{*}{$s_K$}
& 1 &    &\rj{12} &\rj{8} &\rj{24} &\rj{20} &      \\
& 2 &\rj{60}&  \rj{56}&\rj{48}&\rj{80}&\rj{72}&   \rj{68}\\[5pt]\hline\hline
\multirow{2}{*}{$\psi(s_K)$}
& 1 &    &\pj{1}{12}&\pj{19}{24}&\pj{17}{24}&\pj{67}{60} &   \\
& 2 &\pj{3}{20}&\pj{67}{168}&\pj{33}{48}&\pj{163}{240}&\pj{65}{72}&\pj{235}{204}\\
\end{tabular}
\caption{Roots $s_K$ of integer inverses from \eqref{sKs2K} (top) 
and values $\psi(s_K)$ from \eqref{psisK} (bottom) obtained for integer values of $p\ne0$ and $K=1,2$. Some of these values duplicate results from Table~\ref{table1} obtained for other settings of $p$ and $K$.}
\label{table3}
\end{center}
\end{table}
\end{remark}

\section{Comments and extensions}
\label{sec4}
The values of the Dirichlet series defining the second zeta function 
$\psi$ from \eqref{psi} and its anti-periodization $f$ from \eqref{psi} were determined by exploiting similarity relations of the Abel type for functions related to $\psi$ and $f$.
The results provide values at $r=1/\sqrt{n}$ for certain $n$, in particular
for all perfect squares $n$ in the case of $f$. These results include the value $f(1/\sqrt{8})=1$ used in the fluid study \cite{GYWL20}. They do not include, however, cases where $f$ vanishes, which we leave as a conjecture
supported by numerical evidence:
\begin{conjecture}\label{con4.1}
$f(1/\sqrt{n})=0$ if and only if $n=16p(p+1)$ for some $p>0$.
\end{conjecture}
The Abel equations were solved via an iterative approach,
with the periodicity of $\psi$ and $f$ providing closure. 
For $f$, the closure equation \eqref{r0rkany}, $s_0=s_{2K}-p$, 
can be thought of applying $2K$ steps of the map $g$, starting 
from $s_0$, to obtain $s_{2K}$, then applying a shift of $-p$ units 
from $s_{2K}$, to get back to $s_0$. Other results can instead be obtained by alternating applications of the similarity map $g$ and the integer shift, as illustrated by the following Theorem.
\begin{theorem}\label{th4.1}
Let $s_0$ be given and define $s_{k+1}=g(s_k-1)$ for $k=1,2,\ldots$ 
Then 
\begin{equation}
f(s_{k+1}) = s_{k+1}\left(2+\frac{f(s_k)}{1-s_k}\right).
\end{equation}
\end{theorem}
\begin{proof}
The result is a direct consequence of the relation $G(s_{k+1})-G(s_k-1)=1$
obtained from the Abel equation \eqref{Abelf} with $r=s_k-1$ 
and the $1$-antiperiodicity of $f$, $f(s_k-1)=-f(s_k)$.
\end{proof}
All sequences $\{s_k\}_{k\ge0}$ considered in Theorem \eqref{th4.1} converge to $1/2$, regardless of the choice of $s_0$. For $s_0=1/2-1/(2\ell)$, $\ell=1$ or $2$, theorem~\ref{th4.1} yields, after $K$ iterations,
\begin{equation}\label{fsK2}
s_K = \frac{1}{2}-\frac{1}{4K+2\ell}, \quad 
f(s_K) = \frac{1}{4}\cdot\frac{1}{1-2s_K} + (2-\ell)\frac{1-2s_K}{4}.
\end{equation}
For $K=\ell=1$ one recovers \eqref{fonethird}. For large $K$, \eqref{fsK2}
also conforms to the asymptotic estimate around $r=1/2$ used in the model \eqref{rational_approx} of $f$. 
Note that $f(r_K)=\infty$ for $r_K=1/2-1/(4K+2\ell+1))$ with either $\ell=1$ or $\ell=2$. The composition of the unit translation to the left with $g$ maps the interval $[s_0,r_0)=[0,1/6)$ to $[s_1,r_1)=[1/3,5/14)$ ($\ell=1$)
and $[1/4,3/10)$ to $[3/8,7/18)$ ($\ell=2$), where $f$ can be observed to behave similarly in figure~\ref{fig:fsimilar}; more generally, this composition maps $[s_K,r_K)$ to $[s_{K+1},r_{K+1})$ ($\ell=1$ or $2$).

The similarity properties of the function $f$ play a central role in shaping up the fluid response in the fluid study that motivated this work \cite{GYWL20}. Multipole approximations of the form \eqref{rational_approx} may prove useful in understanding the multiscale nature of that response and obtain meaningful numerical approximations.
The approach used here can also be generalized to the study of other trigonometric Dirichlet series such as higher-order secant zeta functions, defined in \cite{LRR14}, and their antiperiodizations/antisymmetrizations.

\appendix
\section{$\psi(1/\sqrt{n})$ for $n=p(p+1)$ with $p$ even}
\label{appA}

Charollois \& Greenberg
provide the explicit expression 
\cite[Eq. 5 with $k=1$ and an additional factor $4/\pi^2$ resulting from the scaling in \eqref{psi}]{ChGr14}
\begin{equation}\label{A1}
\psi(r) = \frac{32}{\lambda(\lambda-1)}\sum_{j=1}^c 
\sum_{\ell=0}^3 \frac{B_\ell(x_j)}{\ell!}\frac{B_{3-\ell}(y_j)}{(3-\ell)!}
(-\lambda)^\ell,
\end{equation}
where $x_j=(j-1/4)/c$, $y_j=\{dx_j\}$ is the fractional part of $dx_j$,
$B_\ell$ is the Bernoulli polynomial of degree $\ell$ \cite[Table 24.2.2]{DLMF}, and $r$, $\lambda$, $c\in\mathbb{Z}$ and $d\in\mathbb{Z}$ are related via the eigenvalue problem
\begin{equation}\label{A2}
V\begin{bmatrix}2r\\1\end{bmatrix} 
= \lambda\begin{bmatrix}2r\\1\end{bmatrix} \quad\text{with}\quad
V=\begin{bmatrix}a&b\\c&d\end{bmatrix}
\end{equation}
for some $a\in\mathbb{Z}$ and $b\in\mathbb{Z}$ such that $ad-bc=1$ ($V\in\text{SL}_2(\mathbb{Z})$). 
For $r=1/\sqrt{n}$ with $n=p(p+1)$, \eqref{A2} holds with
\begin{equation}\label{A3}
V = \begin{bmatrix} 2p+1 & 4 \\ p(p+1) & 2p+1 \end{bmatrix},
\quad \lambda = 2p+1+2\sqrt{p(p+1)},
\end{equation}
but satisfies the assumption $[(2p+1)/4,1]=[1/4,0]V\equiv[1/4,0]$ 
mod[$\mathbb{Z}^2$] used in obtaining \eqref{A1} only when $p$ is even.
The following function implements \eqref{A1} in \textsc{Matlab}:
\begin{lstlisting} 
function psi = CG(p) |\textcolor{green!50!black}{\textbf{\% evaluation of $\psi(r)$ at $r=1/\sqrt{p(p+1)}$ using \eqref{A1})}}|
    n = p*(p+1); r = 1/sqrt(n);
    c = n; d = 2*p+1; 
    lambda = 2*r*c+d; 
    j = 1:c; x = (j-1/4)/c; y = d*x; y = y-floor(y);
    B = @(l,x)bernoulli(l,x)/factorial(l);
    s = 0*j;
    for l = 3:-1:0
        s = B(l,x).*B(3-l,y)-lambda*s; |\textcolor{green!50!black}{\textbf{\% H\"orner's rule}}|
    end
    psi = sum(s)*32/(lambda*(lambda-1));
end
\end{lstlisting}
It is compared with a direct evaluation of the (truncated) series \eqref{psi}:
\begin{lstlisting}
function psi = direct(p)
|\textcolor{green!50!black}{\textbf{\% evaluation of $\psi(r)$ at $r=1/\sqrt{p(p+1)}$ using a truncated series \eqref{psi})}}|
    r = 1/sqrt(p*(p+1));
    n = 1:100000000;
    s = sum(1./(n.^2.*cos(n*pi*r)));
    psi = s*4/pi^2;
end
\end{lstlisting}
As expected, the results match for even $p$ only: \\[5pt]
\setlength{\tabcolsep}{8pt}
\begin{tabular}{llll}
\begin{lstlisting}
>> format rat
>> p = 1;
>> CG(p) 
ans =
   |\textbf{713/605}|

>> direct(p)
ans =
    |\textbf{-5/6}|
\end{lstlisting}
&
\begin{lstlisting}

>> p = 2;
>> CG(p) 
ans =
     |\textbf{3/2}|     

>> direct(p)
ans =
     |\textbf{3/2}|     
\end{lstlisting}
&
\begin{lstlisting}

>> p = 3;
>> CG(p)
ans =
   |\textbf{-130/1351}|  

>> direct(p)
ans =
     |\textbf{1/12}|
\end{lstlisting}
&
\begin{lstlisting}

>> p = 4;
>> CG(p)
and =
    |\textbf{67/60}|    

>> direct(p)
ans =
    |\textbf{67/60}|    
\end{lstlisting}
\end{tabular}

\section{$\psi(1/\sqrt{n})$ for $n=p(p+1)$ revisited}
\label{appB}

The explicit evaluation of $\psi(1/\sqrt{n})$ with $n=p(p+1)$ in Appendix \ref{appA} does not match the result of the formula \eqref{A1} derived in \cite{ChGr14} for odd $p$. The reason is that $[1/4,0]$ is only a left eigenvector, modulo $\mathbb{Z}^2$, of the matrix $V\in\text{SL}_2(\mathbb{Z})$ from \eqref{A2} when $p$ is even. To remedy the situation, it is natural to consider replacing $V$ by 
\begin{equation}\label{B3}
V^2 = \begin{bmatrix} (2p+1)^2+4p(p+1) & 8(2p+1) \\ 2p(p+1)(2p+1) & (2p+1)^2+4p(p+1) \end{bmatrix}.
\end{equation}
Then, clearly, $[2p^2+2p+1/4,4p+2]=[1/4,0]V^2\equiv[1/4,0]$ 
mod[$\mathbb{Z}^2$] for any $p$ and the formula \eqref{A1} now holds for all $p$.
Only the coefficients $c$ and $d$ of the second row of $V$ need to be updated and replaced by those in the second row of $V^2$. The modified \textsc{Matlab} function \text{CG} reads: 
\begin{lstlisting} 
function psi = CG2(p) 
|\textcolor{green!50!black}{\textbf{\% evaluation of $\psi(r)$ at $r=1/\sqrt{p(p+1)}$ using \eqref{A1})}}|
    n = p*(p+1); r = 1/sqrt(n);
    c = n; d = 2*p+1; 
    c = 2*c*d; d = d^2+4*n; |\textcolor{green!50!black}{\textbf{\% added line}}|
    lambda = 2*r*c+d; 
    j = 1:c; x = (j-1/4)/c; y = d*x; y = y-floor(y);
    B = @(l,x)bernoulli(l,x)/factorial(l);
    s = 0*j;
    for l = 3:-1:0
        s = B(l,x).*B(3-l,y)-lambda*s; % |\textcolor{green!50!black}{H\"orner's rule}|
    end
    psi = sum(s)*32/(lambda*(lambda-1));
end
\end{lstlisting}
The results now match for all integers $p$: \\[5pt]
\setlength{\tabcolsep}{15pt}
\begin{tabular}{llll}
\begin{lstlisting}
>> format rat
>> p = 1;
>> CG2(p) 
ans =
    |\textbf{-5/6}|
\end{lstlisting}
&
\begin{lstlisting}

>> p = 2;
>> CG2(p) 
ans =
     |\textbf{3/2}|     
\end{lstlisting}
&
\begin{lstlisting}

>> p = 3;
>> CG2(p)
ans =
     |\textbf{1/12}|
\end{lstlisting}
&
\begin{lstlisting}

>> p = 4;
>> CG2(p)
ans =
    |\textbf{67/60}|    
\end{lstlisting}
\end{tabular} \\[5pt]
Note that the value \texttt{lambda} computed in \texttt{CG2} corresponds to the eigenvalue $\lambda^2$ of $V^2$; it does not need to be updated, since it uses the updated values \texttt{c} and \texttt{d} of $c$ and $d$.

The squaring from $V$ in \eqref{A3} to $V^2$ in \eqref{B3} is essentially the modular equivalent to the strategy presented in \S\ref{sec2} of this work, using periodicity after an even number of steps $2K$ to evaluate $\psi$ at $r_K$. The values obtained for $p$ odd and $p$ even match the columns in Table \ref{table1} associated with $q=-1$ and $q=1$, respectively.
It remains unclear how to select $V$ fullfilling appropriate conditions to extend \eqref{A1} to match other columns of Table \ref{table1} and,
more generally, to other values of $n$. It is also unclear 
how to use \eqref{A1} in order to evaluate \eqref{f} itself for $r=1/\sqrt{p(p+1)}$ via either \eqref{relation2} or \eqref{relation}, since neither $2r$ nor $r+1$ is of the form $1/\sqrt{n'}$ with $n'=p'(p'+1)$ for some integer $p'$.

\section{The value of $f(1/4)$ in \eqref{fonefourth}}
\label{appC}

Using the sine double angle formula and 
\cite[Eq. 24.8.5 \& Table 24.2.2]{DLMF}
\begin{align*}
f(1/4) 
&= \sum_{k\ge0}
\frac{2\sin[(2k+1)\pi/4]}{(2k+1)^2\sin[(2k+1)\pi/2]}
=2\sum_{k\ge0} \frac{(-1)^k\sin[(2k+1)\pi/4]}{(2k+1)^2}\\
&=2\sum_{k\ge0} \frac{\cos[(2k+1)\pi/4]}{(2k+1)^2}
=-2E_1(1/4)=1/2,
\end{align*}
where $E_1(x)=x-1/2$ is the Euler polynomial of degree 1. 
This value can be verified with \textsc{Matlab}:
\begin{lstlisting}
>> format rat
>> k = 1:2:10000000; 
>> f = @(r)sum(1./(k.^2.*cos(k*pi*r)))*4/pi^2;
>> f(1/4)
ans =
   |\textbf{1/2}|
\end{lstlisting}

\section{The approximation \eqref{rational_approx}}
\label{appD}

The Euler infinite product \cite[Eq. 4.22.2]{DLMF}
\begin{equation}
\cos[(2k+1)\pi r] 
= \prod_{\ell\in\mathbb{Z}}\left(1-\frac{2r(2k+1)}{2\ell+1}\right)
\end{equation}
yields the partial fraction decomposition
(equivalent to \cite[Eq. 4.22.5]{DLMF})
\begin{equation}
\frac{1}{\cos[(2k+1)\pi r]}
=\frac{2}{\pi}\sum_{\ell\in\mathbb{Z}} \frac{(-1)^\ell}{2\ell+1-2r(2k+1)}.
\end{equation}
Therefore,
\begin{equation}
f(r)=\frac{8}{\pi^3}\sum_{\ell\in\mathbb{Z}}(-1)^\ell\sum_{k\ge0} \frac{1}{(2k+1)^2}\;\frac{1}{2\ell+1-2r(2k+1)}.
\end{equation}
The behavior at $r_p=1/(4p+2)$ is determined by considering the term 
$k=p(2\ell+1)+\ell$ in the inner sum, which leads to
\begin{equation}
f(r)\sim\frac{1}{(2p+1)^2}\left[\frac{8}{\pi^3}
\sum_{\ell\in\mathbb{Z}}\frac{(-1)^\ell}{(2\ell+1)^3}\right]
\;\frac{1}{1-r(4p+2)} = \frac{r_p^2}{1-r/r_p}
\quad\text{at}\quad r\sim r_p=\frac{1}{4p+2}
\end{equation}
using the value $-E_2(1/4)=1/4$ for the bracketed series
in terms of the Euler polynomial $E_2(x)=x^2-x$ of degree 2
\cite[Eq. 24.8.4 \& Table 24.2.2]{DLMF}.

\vskip6pt

\enlargethispage{20pt}

\begin{ack}
This work has benefited from discussions with the author's colleagues Juan M. Lopez and Jason Yalim in the School of Mathematical \& Statistical Sciences at Arizona State University, who provided insightful comments and suggested improvements to an early draft of the manuscript.
\end{ack}

\vskip2pc

\bibliographystyle{RS}
\bibliography{local}

\end{document}